\documentclass[openright,11pt]{article}

\usepackage{graphicx}
\usepackage[latin1]{inputenc}
\usepackage[english]{babel}
\usepackage{here}
\usepackage{amsmath,amssymb,amscd}
\usepackage{amsthm}
\usepackage{mathrsfs}
\usepackage{enumerate}

\usepackage[flushmargin]{footmisc}
\newcommand\blfootnote[1]{%
	\begingroup
	\renewcommand\thefootnote{}\footnote{#1}%
	\addtocounter{footnote}{-1}%
	\endgroup}

\theoremstyle{plain}
\newtheorem{theorem}{Theorem}[section]
\newtheorem{proposition}[theorem]{Proposition}

\newtheorem{lemma}[theorem]{Lemma}
\newtheorem{corollary}[theorem]{Corollary}

\newtheorem*{claim}{Claim}

\theoremstyle{definition}
\newtheorem{definition}[theorem]{Definition}
\newtheorem{remark}[theorem]{Remark}
\newtheorem{example}[theorem]{Example}
\newtheorem{question}[theorem]{Question}

\newcommand{\NN}{\mathbb{N}}
\newcommand{\ZZ}{\mathbb{Z}}
\newcommand{\RR}{\mathbb{R}}
\newcommand{\CC}{\mathbb{C}}
\newcommand{\KK}{\mathbb{K}}
\newcommand{\DD}{\mathbb{D}}
\newcommand{\TT}{\mathbb{T}}

\newcommand{\Kc}{\mathcal{K}}
\newcommand{\Lc}{\mathcal{L}}

\newcommand{\Uc}{\mathcal{U}}

\newcommand{\Rec}{\textup{Rec}}
\newcommand{\HC}{\textup{HC}}

\newcommand{\cl}{\overline}
\newcommand{\orb}{\textup{Orb}}
\newcommand{\lspan}{\textup{span}}
\newcommand{\res}{\arrowvert}
\newcommand{\eps}{\varepsilon}

\newcommand{\Ran}{\textup{Ran}}
\newcommand{\Ker}{\textup{Ker}}
\newcommand{\codim}{\operatorname{codim}}
\newcommand{\ind}{\textup{ind}}

\newcommand{\til}{\widetilde}

\newcommand{\ep}[2]{\left\langle #1 , #2 \right\rangle}

\newcommand{\QEDh}{\pushQED{\qed}\qedhere} 

\begin{document}
\begin{center}
	\begin{LARGE}
		{\bf Recurrent subspaces in Banach spaces}
	\end{LARGE}
\end{center}

\begin{center}
	\begin{Large}
		by
	\end{Large}
\end{center}

\begin{center}
	\begin{Large}
		Antoni L\'opez-Mart\'inez\blfootnote{\textbf{Journal-ref}: International Mathematics Research Notices, Volume 2024, Issue 11, 2024, Pages 9067--9087.\\ \textbf{DOI}: https://doi.org/10.1093/imrn/rnad321}
	\end{Large}
\end{center}

\vspace*{0.1in}

\begin{abstract}
	We study the {\em spaceability} of the set of recurrent vectors $\Rec(T)$ for an operator $T:X\longrightarrow X$ on a Banach space $X$. In particular: we find sufficient conditions for a quasi-rigid operator to have a recurrent subspace; when $X$ is a complex Banach space we show that having a recurrent subspace is equivalent to the fact that the essential spectrum of the operator intersects the closed unit disc; and we extend the previous result to the real case. As a consequence we obtain that: {\em a weakly-mixing operator on a real or complex separable Banach space has a hypercyclic subspace if and only if it has a recurrent subspace}. The results exposed exhibit a symmetry between the hypercyclic and recurrent spaceability theories showing that, at least for the {\em spaceable} property, hypercyclicity and recurrence can be treated as equals.
\end{abstract}

\vspace*{0.1in}

\section{Introduction}

In {\em Linear Dynamics} the interest in the {\em hypercyclicity} notion comes from the very famous {\em invariant subset problem}, which remains open for operators acting on Hilbert spaces: an operator has no non-trivial invariant closed subset if and only if each non-zero vector is hypercyclic for it. This motivated the study of the structure of the set of hypercyclic vectors, and it is well-known due to Herrero and Bourdon that such a set is always {\em dense lineable}, i.e.\ every hypercyclic operator admits a (non-closed) dense invariant subspace that consists (except for the zero-vector) entirely of hypercyclic vectors (see \cite[Theorem~2.55]{GrPe2011}). The {\em spaceability} (i.e.\ the property of admitting an infinite-dimensional closed subspace) of such a set has also been deeply studied and for case of operators acting on Banach spaces we will follow the basic references \cite{Montes1996,LeMon1997,LeMon2001,GonLeMon2000}, \cite[Chapter 8]{BaMa2009} and \cite[Chapter 10]{GrPe2011}.\\[-5pt]

In its turn, {\em recurrence} is one of the fundamental concepts in dynamics since the beginning of the theory at the end of the 19th century when Poincar\'e proved his Recurrence Theorem, even though most of the literature in the context of Linear Dynamics is built around the central concept of hypercyclicity. In fact, we refer to the 2014 paper of Costakis, Manoussos and Parissis \cite{CoMaPa2014} as the beginning of the systematic study of {\em linear recurrence}, despite the great non-linear dynamical knowledge already existing in this area (see \cite{Furstenberg1981}). The linear structure of the set of recurrent vectors has been recently studied, and many results about {\em lineability} and {\em dense lineability} have been obtained, see \cite{GriLoPe2022}.\\[-5pt]

The aim of this paper is rewriting the already well-known {\em hypercyclic spaceability theory} on Banach spaces for recurrence and, moreover, linking both concepts when we consider weakly-mixing operators.

The paper is organized as follows. In Section~\ref{Sec:2Notation} we introduce the notation, basic concepts and the historical development of the {\em hypercyclic spaceability theory} on Banach spaces. In Section~\ref{Sec:3RecSub} we present the obtained results about the spaceability of the recurrent vectors by establishing a symmetry with the hypercyclic case. Section~\ref{Sec:4Basic} is devoted to prove sufficient conditions for a quasi-rigid operator to have a {\em recurrent subspace}. In Section~\ref{Sec:5complex} we study the essential spectrum for recurrent operators and its consequences on the existence of recurrent subspaces when the underlying space is complex. Finally, Section~\ref{Sec:6real} exhibits the real case of the previous results while in Section~\ref{Sec:7applications} we investigate their applications and we show that every {\em C-type operator} as defined in \cite{Menet2017,GriMaMe2021} has a {\em hypercyclic subspace}.\\[-5pt]

We refer to the textbooks \cite{BaMa2009,GrPe2011} for any Linear Dynamics' unexplained but standard notion.

\section{Notation and basic concepts}\label{Sec:2Notation}

\subsection{General background}

Given an infinite-dimensional Banach space $X$ and a subset of vectors $Y \subset X$ with some property, we say that $Y$ is {\em spaceable} if there is an infinite-dimensional closed subspace $Z \subset X$ such that $Z\setminus\{0\} \subset Y$. In this paper we study the {\em spaceability} property for some subsets of vectors with certain dynamical properties (hypercyclicity and recurrence) with respect to a bounded linear operator.\\[-5pt]

From now on let $T:X\longrightarrow X$ be a {\em bounded linear operator} on a ({\em real} or {\em complex}) {\em separable infinite-dimensional Banach space} $X$, and denote by $\Lc(X)$ the {\em set of bounded linear operators} acting on such a space $X$. Given a vector $x \in X$ we denote its {\em $T$-orbit} by
\[
\orb(x,T) := \{ T^nx : n \in \NN_0 \},
\]
where $\NN_0 = \NN \cup \{0\}$. We are interested in the dynamical properties:
\begin{enumerate}[(a)]
	\item {\em hypercyclicity}: a vector $x \in X$ is {\em hypercyclic} for $T$ if its $T$-orbit is dense in $X$. We denote by $\HC(T)$ the set of such vectors, and we $T$ is called {\em hypercyclic} if $\HC(T)$ is non-empty;
	
	\item {\em recurrence}: a vector $x \in X$ is {\em recurrent} for $T$ if there exists an increasing sequence $(k_n)_{n\in\NN}$ such that $T^{k_n}x \to x$. We denote by $\Rec(T)$ the set of such vectors, and we say that $T$ is {\em recurrent} if $\Rec(T)$ is dense in $X$.
\end{enumerate}

It is worth highlighting that for an operator $T \in \Lc(X)$ to be hypercyclic it is enough to admit just one hypercyclic vector, while to be recurrent we demand the existence of a dense set of recurrent vectors. This is the natural definition for recurrence since every operator admits the zero-vector as a fixed point (hence recurrent) so that the existence of just one recurrent vector is not enough to talk about a ``global'' recurrent behaviour, as it has been discussed in the very recent recurrence works \cite{CoMaPa2014,BesMePePu2016,BoGrLoPe2022,CarMur2022_MS,CarMur2022_arXiv,GriLo2023,GriLoPe2022}. Following the hypercyclicity works \cite{LeMon2001,GonLeMon2000}, \cite[Chapter 8]{BaMa2009} and \cite[Chapter 10]{GrPe2011} we will use the notation:

\begin{definition}
	Let $X$ be a (real or complex) Banach space and let $T \in \Lc(X)$:
	\begin{enumerate}[(a)]
		\item a {\em hypercyclic subspace} for $T$ is an infinite-dimensional closed subspace $Z \subset X$ such that $Z \setminus \{0\} \subset \HC(T)$;
		
		\item a {\em recurrent subspace} for $T$ is an infinite-dimensional closed subspace $Z \subset X$ such that $Z \subset \Rec(T)$.
	\end{enumerate}
\end{definition}

In the next subsection we recall the existent {\em hypercyclic spaceability theory} for operator acting on Banach spaces, pointing out the similarities that we will find in the recurrent case.

\subsection{Hypercyclic subspaces}

The first known result, which established sufficient conditions for an operator to have a hypercyclic subspace, is due to A. Montes-Rodr\'iguez \cite{Montes1996}:

\begin{theorem}[\textbf{\cite{Montes1996}}]\label{The:Montes}
	Let $X$ be a (real or complex) separable Banach space and let $T \in \Lc(X)$. Assume that there is an increasing sequence of integers $(k_n)_{n\in\NN}$ such that:
	\begin{enumerate}[{\em(a)}]
		\item $T$ satisfies the Hypercyclicity Criterion with respect to $(k_n)_{n\in\NN}$;
		
		\item there is an infinite-dimensional closed subspace $E \subset X$ such that $T^{k_n}x \to 0$ for all $x \in E$.
	\end{enumerate}
	Then $T$ has a hypercyclic subspace.
\end{theorem}

One may note various things when looking at Theorem~\ref{The:Montes}:
\begin{enumerate}[--]
	\item the original statement used the whole sequence $(k_n)_{n\in\NN} = (n)_{n\in\NN}$, i.e.\ $T$ was required to satisfy the Kitai's Criterion instead of the actual more refined Hypercyclicity Criterion (see \cite{LeMu2006} or \cite[Theorem 8.1]{BaMa2009} for Theorem~\ref{The:Montes} stated here);
	
	\item it is now well-known (due to J. B\`es and A. Peris \cite{BesPe1999}) that an operator $T$ satisfies the Hypercyclicity Criterion if and only if it is {\em weakly-mixing}, i.e.\ if and only if the $N$-fold direct sum operator $T\oplus\cdots\oplus T$ is hypercyclic for every $N \in \NN$;
	
	\item however, the sequence with respect to which conditions (a) and (b) are satisfied must coincide, so a priori one cannot exchange the hypothesis that {\em $T$ satisfies the Hypercyclicity Criterion with respect to $(k_n)_{n\in\NN}$}, by the apparently equivalent {\em $T$ is weakly-mixing};
	
	\item there exists a second proof of Theorem~\ref{The:Montes} that relies on a really nice idea of K. C. Chan (see \cite{Chan1999,Chan2001}). We will not use that approach in this document, but it is worth mentioning that it also needs both (a) and (b) conditions with respect to the same sequence $(k_n)_{n\in\NN}$.
\end{enumerate}

The following result that appeared about the existence of hypercyclic subspaces for operators on Banach spaces is due to F. Le\'on-Saavedra and A. Montes-Rodr\'iguez \cite{LeMon1997}. They showed that Theorem~\ref{The:Montes} applies for weakly-mixing compact perturbations of the identity:

\begin{theorem}[\textbf{\cite{LeMon1997}}]\label{The:LeMon}
	Let $X$ be a (real or complex) separable Banach space and let $T \in \Lc(X)$. Assume that $T$ is weakly-mixing and that there exists some compact operator $K$ such that $\|T-K\| \leq 1 $. Then $T$ has a hypercyclic subspace.
\end{theorem}

Note that, contrary to Theorem~\ref{The:Montes}, no assumption on the sequence with respect to which $T$ satisfies the Hypercyclicity Criterion is necessary for Theorem~\ref{The:LeMon} and, indeed, we have stated the result just including the {\em weak-mixing} assumption. The development of the theory followed by giving a complete characterization of the operators that admit a hypercyclic subspace, provided they are weakly-mixing. This was first established by F.~Le\'on-Saavedra and A.~Montes-Rodr\'iguez for Hilbert spaces \cite{LeMon2001}, and later by M.~Gonz\'alez and the previous two authors for the general Banach space case \cite{GonLeMon2000}.\\[-5pt]

To state the announced characterization we move into the {\em complex} setting, and in fact, we need the concept of {\em essential spectrum} for an operator acting on a complex Banach space (see~\cite{Conway1989}). We also introduce the concept of {\em left-essential spectrum}, since it turns out to be an important tool for the result, and we also use it along the paper (see Section~\ref{Sec:5complex}):

\begin{definition}
	Let $X$ be a complex Banach space and let $\Kc(X)$ be the two-sided ideal of $\Lc(X)$ consisting of all compact operators on $X$. Given $T \in \Lc(X)$ we denote by $[T]_{\Lc/\Kc}$ the image of the operator $T$ under the standard projection
	\[
	\Lc(X) \hookrightarrow \Lc(X)/\Kc(X),
	\]
	where the Banach algebra $\Lc(X)/\Kc(X)$ is the known {\em Calkin algebra}. Then:
	\begin{enumerate}[(a)]
		\item the {\em left-essential spectrum} of $T$ is the compact set of complex numbers
		\[
		\sigma_{\ell e}(T) = \left\{ \lambda \in \CC : [T-\lambda]_{\Lc/\Kc} \text{ is not left-invertible in } \Lc(X)/\Kc(X) \right\};
		\]
		
		\item the {\em essential spectrum} of $T$ is the compact set of complex numbers
		\[
		\sigma_e(T) = \left\{ \lambda \in \CC : [T-\lambda]_{\Lc/\Kc} \text{ is not invertible in } \Lc(X)/\Kc(X) \right\}.
		\]
	\end{enumerate}
\end{definition}

In Section~\ref{Sec:5complex} we examinate the previous concepts from the {\em Fredholm} and {\em semi-Fredholm theory} perspective. For the moment we are able to state the mentioned characterization: 

\begin{theorem}[\textbf{\cite{GonLeMon2000}}]\label{The:GonLeMon}
	Let $X$ be a complex separable Banach space and let $T \in \Lc(X)$. If $T$ is weakly-mixing, then the following statements are equivalent:
	\begin{enumerate}[{\em(i)}]
		\item $T$ has a hypercyclic subspace;
		
		\item there exists an infinite-dimensional closed subspace $E \subset X$ and an increasing sequence of integers $(l_n)_{n\in\NN}$ such that $T^{l_n}x \to 0$ for all $x \in E$;
		
		\item there exists an infinite-dimensional closed subspace $E \subset X$ and an increasing sequence of integers $(l_n)_{n\in\NN}$ such that $\sup_{n\in\NN} \|T^{l_n}\res_E\| < \infty$;
		
		\item the essential spectrum of $T$ intersects the closed unit disk, i.e.\ $\sigma_e(T) \cap \cl{\DD} \neq \varnothing$.
	\end{enumerate}
\end{theorem}

Theorem~\ref{The:GonLeMon} contains the {\em complex} cases of Theorems~\ref{The:Montes} and \ref{The:LeMon}, the second being true since the compact perturbation of any operator preserves its essential spectrum. As we have already mentioned the {\em weakly-mixing} assumption is a really important tool in the three stated results, but to conclude the existence of a hypercyclic subspace there is always a second assumption about the ``regularity'' of $T$ on some subspace of $X$.

\section{Recurrent subspaces}\label{Sec:3RecSub}

In order to rewrite the previous theory for recurrence, we need to introduce the analogous ingredients to those observed in Theorems~\ref{The:Montes}, \ref{The:LeMon} and \ref{The:GonLeMon}. First of all we need the operator $T \in \Lc(X)$ to have a recurrence property similar to the {\em weak-mixing} assumption:
\begin{enumerate}[--]
	\item {\em the $N$-fold direct sum operator $T\oplus\cdots\oplus T$ has to be recurrent for every $N \in \NN$}.
\end{enumerate}
It has been recently shown that, for separable Banach spaces, this property is equivalent to the {\em quasi-rigidity} notion (see \cite{GriLoPe2022}):

\begin{definition}[\textbf{\cite{GriLoPe2022}}]
	Let $X$ be a Banach space. We say that $T \in \Lc(X)$ is a {\em quasi-rigid} operator if there exists an increasing sequence $(k_n)_{n\in\NN}$ and a dense subset $Y \subset X$ such that $T^{k_n}x \to x$ for all $x \in Y$. In this case we say that $T$ is {\em quasi-rigid with respect to} $(k_n)_{n\in\NN}$.
\end{definition}

Secondly, we need an assumption about the ``regularity'' of the operator $T$. In Theorem~\ref{The:Montes} this assumption was condition (b). However, since we would like the identity operator $I:X\longrightarrow X$ to fulfill the main theorems of the {\em recurrent spaceability theory}, we will use a more realistic condition than that of having, for a recurrent (and not necessarily hypercyclic) operator, an infinite-dimensional closed subspace of vectors with $0$-convergent suborbits:

\begin{theorem}\label{The:Banach-sufficient}
	Let $X$ be a (real or complex) separable Banach space and let $T \in \Lc(X)$. Assume that there is an increasing sequence of integers $(k_n)_{n\in\NN}$ such that:
	\begin{enumerate}[{\em(a)}]
		\item $T$ is quasi-rigid with respect to $(k_n)_{n\in\NN}$;
		
		\item there exists a non-increasing sequence $(E_n)_{n\in\NN}$ of infinite-dimensional closed subspaces of $X$ such that $\sup_{n\in\NN} \left\|T^{k_n}\res_{E_n}\right\| < \infty$.
	\end{enumerate}
	Then $T$ has a recurrent subspace. In particular, there exists an infinite-dimensional closed subspace $F \subset X$ and a subsequence $(l_n)_{n\in\NN}$ of $(k_n)_{n\in\NN}$ such that $T^{l_n}x \to x$ for all $x \in F$, so
	\[
	F\oplus\cdots\oplus F \subset \Rec(T\oplus\cdots\oplus T)
	\]
	for every $N$-fold direct sum operator $T\oplus\cdots\oplus T : X\oplus\cdots\oplus X \longrightarrow X\oplus\cdots\oplus X$.
\end{theorem}

Theorem~\ref{The:Banach-sufficient} can be compared with Theorem~\ref{The:Montes}. Note that condition (b) of Theorem~\ref{The:Banach-sufficient} is satisfied by the identity operator and that, for recurrent (and not necessarily hypercyclic) operators, it is a more suitable assumption than condition (b) of Theorem~\ref{The:Montes}. The proof of Theorem~\ref{The:Banach-sufficient} (see Section~\ref{Sec:4Basic}) is highly based on the arguments and {\em basic sequence} techniques used in the proof of Theorem~\ref{The:Montes}. In particular, our proof can be compared with the known results \cite[Theorem~20]{LeMu2006} or \cite[Theorem~10.29]{GrPe2011}, where condition (b) of Theorem~\ref{The:Banach-sufficient} was already used to obtain the existence of hypercyclic subspaces.\\[-5pt]

An application of the Banach-Steinhaus theorem together with Theorem~\ref{The:Banach-sufficient} yields that, {\em given a (real or complex) Banach space $X$ and given $T \in \Lc(X)$ a quasi-rigid operator with respect to the sequence $(k_n)_{n\in\NN}$, the following conditions are equivalent:
\begin{enumerate}[--]
	\item there exists an infinite-dimensional closed subspace $E \subset X$ and a subsequence $(l_n)_{n\in\NN}$ of $(k_n)_{n\in\NN}$ such that $T^{l_n}x \to x$ for all $x \in E$;
	
	\item there exists an infinite-dimensional closed subspace $E \subset X$ and a subsequence $(l_n)_{n\in\NN}$ of $(k_n)_{n\in\NN}$ such that $(T^{l_n}x)_{n\in\NN}$ converges in $X$ for all $x \in E$;
	
	\item there exists an infinite-dimensional closed subspace $E \subset X$ and a subsequence $(l_n)_{n\in\NN}$ of $(k_n)_{n\in\NN}$ such that $(T^{l_n}x)_{n\in\NN}$ is bounded in $X$ for all $x \in E$;
	
	\item there exists an infinite-dimensional closed subspace $E \subset X$ and a subsequence $(l_n)_{n\in\NN}$ of $(k_n)_{n\in\NN}$ such that $\sup_{n\in\NN} \|T^{l_n}\res_E\| < \infty$;
	
	\item there exists a non-increasing sequence $(E_n)_{n\in\NN}$ of infinite-dimensional closed subspaces of $X$ and a subsequence $(l_n)_{n\in\NN}$ of $(k_n)_{n\in\NN}$ such that $\sup_{n\in\NN} \left\|T^{l_n}\res_{E_n}\right\| < \infty$;
\end{enumerate}
and if any of them holds, then $T$ has a recurrent subspace}.\\[-5pt]

A stronger result is in fact true: by following the ideas of \cite{LeMon2001,GonLeMon2000} and \cite[Chapter 8]{BaMa2009} we can establish the equivalence between the previous statements and the (apparently weaker) condition of admitting a recurrent subspace. Indeed, our next result can be compared with the previously stated Theorem~\ref{The:GonLeMon}:

\begin{theorem}\label{The:Banach-complex}
	Let $X$ be a complex separable Banach space and let $T \in \Lc(X)$. If $T$ is quasi-rigid, then the following statements are equivalent: 
	\begin{enumerate}[{\em(i)}]
		\item $T$ has a recurrent subspace;
		
		\item there exists an infinite-dimensional closed subspace $E \subset X$ and an increasing sequence of integers $(l_n)_{n\in\NN}$ such that $T^{l_n}x \to x$ for all $x \in E$;
		
		\item there exists an infinite-dimensional closed subspace $E \subset X$ and an increasing sequence of integers $(l_n)_{n\in\NN}$ such that $\sup_{n\in\NN} \|T^{l_n}\res_E\| < \infty$;
		
		\item the essential spectrum of $T$ intersects the closed unit disk, i.e.\ $\sigma_e(T) \cap \cl{\DD} \neq \varnothing$.
	\end{enumerate}
\end{theorem}

Theorem~\ref{The:Banach-complex} reaffirms the fact that {\em quasi-rigidity} is, for recurrence, the analogous property to that of {\em weak-mixing} for hypercyclicity. The proof of Theorem~\ref{The:Banach-complex} (see Section~\ref{Sec:5complex}) is highly based on the proof of Theorem~\ref{The:GonLeMon}, but we shall study the structure of the essential spectrum of recurrent operators in order to complete the different implications required.\\

In view of Theorems~\ref{The:GonLeMon} and \ref{The:Banach-complex} we deduce that: {\em a weakly-mixing operator on a complex separable Banach space has a hypercyclic subspace if and only if it has a recurrent subspace}. However, we can prove that equivalence also for {\em real-linear} operators acting on {\em real} Banach spaces by using some {\em complexification} techniques (see \cite{MuSaTo1999,MoMuPeSe2022}): given a real Banach space $(X,\|\cdot\|)$, the {\em complexification} $\til{X}$ of $X$ is defined formally as the vector space
\[
\til{X} := \left\{ x + iy : x,y \in X \right\},
\]
which can be (algebraically and topologically) identified with $X\oplus X$. Indeed, if the multiplication by complex scalars is defined as
\[
(a+ib)(x+iy) = (ax-by)+i(ay+bx),
\]
for any $a,b \in \RR$ and $x,y \in X$, then $\til{X}$ becomes a complex Banach space with the norm
\[
\|x+iy\|_c := \sup_{t \in [0,2\pi]} \|\cos(t) x  - \sin(t) y\|.
\]
It is easy to check that $\|\cdot\|_c$ is a $\CC$-homogeneous norm that endows $\til{X} = X + iX$ with an homeomorphic topology to that of the usual direct sum space $X\oplus X$, and moreover, the map
\[
J : \til{X} \longrightarrow X\oplus X \quad \text{ with } \quad J(x+iy) = (x,y) \in X\oplus X,
\]
is an $\RR$-isomorphism.\\[-5pt]

Furthermore, given a (real-linear) operator $T:X\longrightarrow X$ on the real Banach space $X$, its {\em complexification} $\til{T}:\til{X}\longrightarrow\til{X}$ is defined by
\[
\til{T}(x+iy) = Tx + iTy \quad \text{ for every } x,y \in X,
\]
which is a (complex-linear) operator on the (complex) Banach space $\til{X}$. It is easily seen that $\til{T}$ is conjugate to $T\oplus T$ via $J$, i.e.\ $J \circ \til{T} = T\oplus T \circ J$, and also that $\|T\|=\|\til{T}\|$.\\

By using this complex structure we can state the real version of Theorems~\ref{The:GonLeMon} and \ref{The:Banach-complex}:

\begin{theorem}\label{The:Banach-real}
	Let $X$ be a real separable Banach space and let $T \in \Lc(X)$:
	\begin{enumerate}[{\em(a)}]
		\item If $T$ is weakly-mixing, then the following statements are equivalent:
		\begin{enumerate}[{\em(i)}]
			\item $T$ has a hypercyclic subspace;
			
			\item there exists an infinite-dimensional closed subspace $E \subset X$ and an increasing \newline sequence of integers $(l_n)_{n\in\NN}$ such that $T^{l_n}x \to 0$ for all $x \in E$;
			
			\item there exists an infinite-dimensional closed subspace $E \subset X$ and an increasing \newline sequence of integers $(l_n)_{n\in\NN}$ such that $\sup_{n\in\NN} \|T^{l_n}\res_E\| < \infty$;
			
			\item the essential spectrum of the complexification $\til{T}:\til{X}\longrightarrow\til{X}$ intersects the closed \newline unit disk, i.e.\ $\sigma_e(\til{T}) \cap \cl{\DD} \neq \varnothing$.
		\end{enumerate}
	
		\item If $T$ is quasi-rigid, then the following statements are equivalent:
		\begin{enumerate}[{\em(i)}]
			\item $T$ has a recurrent subspace;
			
			\item there exists an infinite-dimensional closed subspace $E \subset X$ and an increasing \newline sequence of integers $(l_n)_{n\in\NN}$ such that $T^{l_n}x \to x$ for all $x \in E$;
			
			\item there exists an infinite-dimensional closed subspace $E \subset X$ and an increasing \newline sequence of integers $(l_n)_{n\in\NN}$ such that $\sup_{n\in\NN} \|T^{l_n}\res_E\| < \infty$;
			
			\item the essential spectrum of the complexification $\til{T}:\til{X}\longrightarrow\til{X}$ intersects the closed \newline unit disk, i.e.\ $\sigma_e(\til{T}) \cap \cl{\DD} \neq \varnothing$.
		\end{enumerate}
	\end{enumerate}
\end{theorem}

Theorem~\ref{The:Banach-real} is a modest extension of the previous theory to the real setting. Its proof is highly based on the proofs of Theorems~\ref{The:GonLeMon} and \ref{The:Banach-complex}, but we require some basic lemmas in order to pass from subspaces for the real-linear system to subspaces for its complexification and vice versa, see Section \ref{Sec:6real}.\\[-5pt]

Since the weak-mixing property implies quasi-rigidity for every operator on a separable Banach space (see \cite[Theorem~2.5 and Remark~2.6]{GriLoPe2022}), from Theorems~\ref{The:GonLeMon}, \ref{The:Banach-complex} and \ref{The:Banach-real} we obtain the announced result:

\begin{corollary}\label{Cor:hyp.iff.rec}
	Let $T:X\longrightarrow X$ be a weakly-mixing operator on a (real or complex) separable Banach space $X$. Then the following statements are equivalent:
	\begin{enumerate}[{\em(i)}]
		\item $T$ has a hypercyclic subspace;
		
		\item $T$ has a recurrent subspace.
	\end{enumerate}
\end{corollary}

\section{Sufficient conditions for recurrent subspaces}\label{Sec:4Basic}

In this section we prove Theorem~\ref{The:Banach-sufficient}. As in the proof of Theorem~\ref{The:Montes}:
\begin{enumerate}[--]
	\item we can extract a basic sequence from the family $(E_n)_{n\in\NN}$ of infinite-dimensional closed subspaces, by using the Mazur's theorem (see \cite[Vol I, page 4]{LindTzaf1977} or \cite[Lemma C.1.1]{BaMa2009});
	
	\item we will ``perturb'' that initial basic sequence to obtain an equivalent one formed by vectors with a strong recurrent property (see \cite[Lemmas 8.4 and C.1.2]{BaMa2009} or \cite[Lemma 10.6]{GrPe2011});
\end{enumerate}
We refer to the textbooks \cite{Diestel1984,LindTzaf1977} for any unexplained but standard notion about Schauder basis and basic sequences.

\subsection{Proof of Theorem~\ref{The:Banach-sufficient}}

By the Mazur's theorem (see \cite[Vol I, page 4]{LindTzaf1977} or \cite[Lemma C.1.1]{BaMa2009}) there is a normalized basic sequence $(e_n)_{n\in\NN}$ such that $e_n \in E_n$ for each $n \in \NN$, so it is a Schauder basis of $E:=\cl{\lspan\{ e_n : n \in\NN \}}$. Moreover, for every strictly increasing sequence of integers $(l_n)_{n\in\NN}$, the sequence $(e_{l_n})_{n\in\NN}$ is a normalized Schauder basis of the closed subspace
\[
\cl{\lspan\left\{ e_{l_n} : n \in \NN \right\}} \subset E \subset X.
\]
For each $n \in \NN$ consider the {\em coefficient functional} $e_n^* : E \longrightarrow \KK$ such that for each $x \in E$ with $x=\sum_{n\in\NN} a_n e_n$ then $\ep{e_n^*}{x} = \ep{e_n^*}{\sum_{k\in\NN} a_k e_k} = a_n$, which are uniformly continuous since $(e_n)_{n\in\NN}$ is normalized (see \cite[Vol. I, 1.b]{LindTzaf1977} or \cite[Appendix C.1]{BaMa2009}). Denote by $\|e_n^*\|$ the norm of $e_n^*$ as a functional in $E$ and write $K := 1 + \max_{n\in\NN} \|e_n^*\|$.\\[-5pt]

By condition (a) there is a dense subset $Y$ of $X$ such that $T^{k_n}x \to x$ for every $x \in Y$.
\begin{claim}
	There exists:
	\begin{enumerate}[{\em1)}]
		\item an increasing sequence of integers $(l_n)_{n\in\NN}$, subsequence of $(k_n)_{n\in\NN}$;
		
		\item a sequence of vectors $(f_{l_n})_{n\in\NN} \subset Y$, and the related sequence $(g_{l_n})_{n\in\NN} := (f_{l_n}-e_{l_n})_{n\in\NN}$;
	\end{enumerate}
	with the properties:
	\begin{enumerate}[{\em(i)}]
		\item $\|f_{l_n}-e_{l_n}\| = \|g_{l_n}\| < \frac{1}{2^{n+1}K}$ for every $n \in \NN$;
	
		\item $\|T^{l_j}g_{l_n}\| < \frac{1}{2^{j+n}}$ for each $j \in \NN$ and $n > j$;
		
		\item $\|T^{l_j}f_{l_n} - f_{l_n}\| < \frac{1}{2^{j+n}}$ for each $j \in \NN$ and $1\leq n\leq j$.
	\end{enumerate}
\end{claim}
\begin{proof}[Proof of the Claim]
	Suppose that we have constructed $(l_n)_{n=1}^i$ and $(f_{l_n})_{n=1}^i$ with:
	\begin{enumerate}[(i)]		
		\item $\|f_{l_n}-e_{l_n}\| = \|g_{l_n}\| < \frac{1}{2^{n+1}K}$ for each $n \leq i$;
		
		\item $\|T^{l_j}g_{l_n}\| < \frac{1}{2^{j+n}}$ for each $1\leq j< i$ and $i\geq n > j$;
		
		\item $\|T^{l_j}f_{l_n}-f_{l_n}\| < \frac{1}{2^{j+n}}$ for each $1\leq j \leq i$ and $1\leq n\leq j$.
	\end{enumerate}
	By the continuity of $T$ there is $\eps>0$ such that
	\[
	\|T^{l_j}y\| < \frac{1}{2^{j+(i+1)}} \quad \text{ for every } 1\leq j\leq i \text{ and } y \in X \text{ with } \|y\|<\eps.
	\]
	Therefore, taking $f_{l_{i+1}} \in Y$ such that $\|f_{l_{i+1}}-e_{l_{i+1}}\| < \max\left\{ \frac{1}{2^{i+2}K} , \eps \right\}$ we get 
	\begin{enumerate}[(i)]
		\item $\|f_{l_{i+1}}-e_{l_{i+1}}\| = \|g_{l_{i+1}}\| < \frac{1}{2^{(i+1)+1}K}$, which is (i) for $i+1$;
		
		\item $\|T^{l_j}g_{l_{i+1}}\| < \frac{1}{2^{j+(i+1)}}$ for each $1\leq j< i+1$, which is (ii) for $i+1$.
	\end{enumerate}
	Choose $l_{i+1} \in \{ k_j : j \in \NN \}$ large enough such that $l_{i+1} > l_i$ and $\|T^{l_{i+1}}f_{l_n}-f_{l_n}\| < \frac{1}{2^{(i+1)+n}}$ for every $1\leq n\leq i+1$. This is possible since $f_{l_n} \in Y$, and it implies (iii) for $i+1$.
\end{proof}

Once the {\em Claim} is proved, by condition (i) we get that
\[
\sum_{n\in\NN} \|e_{l_n}^*\| \cdot \|f_{l_n}-e_{l_n}\| = \sum_{n\in\NN} \|e_{l_n}^*\| \cdot \|g_{l_n}\| \overset{\text{(i)}}{<} \sum_{n\in\NN} \frac{1}{2^{n+1}} = \frac{1}{2} < 1,
\]
so \cite[Lemmas 8.4 and C.1.2]{BaMa2009} or \cite[Lemma 10.6]{GrPe2011} imply that $(f_{l_n})_{n\in\NN}$ is a basic sequence equivalent to $(e_{l_n})_{n\in\NN}$. It follows that $(f_{l_n})_{n\in\NN}$ is a Schauder basis of
\[
F := \cl{\lspan\{ f_{l_n} : n \in \NN \}} \subset X,
\]
which is an infinite-dimensional closed subspace of $X$. We claim that $T^{l_j}x \to x$ for all $x \in F$: indeed, given $x \in F$ there is $a = (a_n)_{n\in\NN} \in c_0(\NN)$ such that
\[
x = \sum_{n\in\NN} a_n f_{l_n} = \sum_{n\in\NN} a_n (e_{l_n}+g_{l_n}) = \sum_{n\in\NN} a_n e_{l_n} + \sum_{n\in\NN} a_n g_{l_n},
\]
where $\sum_{n\in\NN} a_n e_{l_n}$ is convergent since $(f_{l_n})_{n\in\NN}$ and $(e_{l_n})_{n\in\NN}$ are equivalent basic sequences, and $\sum_{n\in\NN} a_n g_{l_n}$ is an absolutely convergent series by (i). Hence
\begin{eqnarray}
	\hspace{-0.4cm}\|T^{l_j}x - x\| &=& \left\| \left( \sum_{n\leq j} a_n T^{l_j}f_{l_n} \right) + T^{l_j}\left( \sum_{n>j} a_n g_{l_n} \right) + T^{l_j}\left( \sum_{n>j} a_n e_{l_n} \right) - \sum_{n\in\NN} a_n f_{l_n}\right\| \nonumber\\[10pt]
	&\leq& \left\| \sum_{n\leq j} a_n (T^{l_j}f_{l_n} - f_{l_n}) \right\| + \left\| \sum_{n>j} a_n T^{l_j}g_{l_n} \right\| + \left\|T^{l_j}\left( \sum_{n>j} a_n e_{l_n} \right)\right\| + \left\| \sum_{n>j} a_n f_{l_n} \right\| \nonumber\\[10pt]
	&\leq& \|a\|_{\infty} \left( \sum_{n\leq j} \|T^{l_j}f_{l_n} - f_{l_n}\| + \sum_{n>j} \|T^{l_j}g_{l_n}\| \right) \nonumber\\[10pt]
	&\ & + \ \left\|T^{l_j}\res_{E_{l_{j+1}}}\right\| \cdot \left\| \sum_{n>j} a_n e_{l_n} \right\| + \left\| \sum_{n>j} a_n f_{l_n} \right\| \nonumber\\[10pt]
	&\underset{\text{(b)}}{\overset{\text{(ii),(iii)}}{<}}& \|a\|_{\infty} \sum_{i>j} \frac{1}{2^i} + \sup_{n\in\NN} \left\|T^{k_n}\res_{E_n}\right\| \cdot \left\| \sum_{n>j} a_n e_{l_n} \right\| + \left\| \sum_{n>j} a_n f_{l_n} \right\| \longrightarrow 0, \nonumber
\end{eqnarray}
when $j \to \infty$ since $(e_{l_n})_{n\in\NN}$ and $(f_{l_n})_{n\in\NN}$ are basic sequences.\QEDh

\subsection{Comments on Theorem~\ref{The:Banach-sufficient}}\label{SubSec:4.2comments}

The previous proof allows us to relax the hypothesis of Theorem~\ref{The:Banach-sufficient} in two different ways: we can delete the {\em separability} hypothesis; and we can assume just {\em local quasi-rigidity}:

\begin{remark}[\textbf{Non-Separability}]\label{Rem:non-separability.1}
	To study hypercyclicity it is crucial to assume the {\em separability} of the underlying space, and in fact, Theorem~\ref{The:Montes} applies to bounded linear operators on {\em separable} Banach spaces. However, as it is mentioned in \cite{CoMaPa2014}, the {\em separability} is not a necessary assumption for recurrence. Note that Theorem~\ref{The:Banach-sufficient} is still true, with the same proof, if the underlying Banach space is non-separable.
\end{remark}

\begin{remark}[\textbf{Local Quasi-Rigidity is allowed in Theorem~\ref{The:Banach-sufficient}}]\label{Rem:local.q-r.1}
	Let $(E_n)_{n\in\NN}$ be the sequence of subspaces from condition (b) of Theorem~\ref{The:Banach-sufficient}. The previous proof still holds if we replace condition (a) by the following:
	\begin{enumerate}[(a$^*$)]
		\item {\em there is a set of vectors $Y \subset X$ with $E_1 \subset \cl{Y}$ such that $T^{k_n}x \to x$ for all $x \in Y$}.
	\end{enumerate}
	This last condition is a kind of {\em local quasi-rigidity} and, contrary to the hypercyclicity case, an operator does not need to be recurrent in order to have a recurrent subspace since $T \in \Lc(X)$ is hypercyclic as soon as $\HC(T)\neq\varnothing$ while the recurrence property demands that $\Rec(T)$ is dense (see the definition of recurrence used in this paper and the works \cite{CoMaPa2014,BesMePePu2016,BoGrLoPe2022,CarMur2022_MS,CarMur2022_arXiv,GriLo2023,GriLoPe2022}). Indeed, given $\lambda \in \RR$ with $|\lambda|\neq 1$ we can consider the operator
	\[
	T := I \oplus \lambda I : X\oplus X \longrightarrow X\oplus X,
	\]
	where $I:X\longrightarrow X$ is the identity operator on a Banach space $X$. Clearly $T$ is not recurrent (i.e.\ the set of recurrent vectors $\Rec(T)$ is not dense) but it contains a recurrent subspace. Note that $T$ fulfills property (a$^*$) for any sequence $(k_n)_{n\in\NN}$ whenever $E_1 \subset X\oplus\{0\}$.
\end{remark}

\section{The complex case}\label{Sec:5complex}

In this section the underlying Banach space $X$ will be complex. We prove Theorem~\ref{The:Banach-complex} characterising the quasi-rigid operators admitting a recurrent subspace in a way similar to Theorem~\ref{The:GonLeMon} for weakly-mixing operators admitting a hypercyclic subspace. We start by studying the essential spectrum for recurrent operators.

\subsection{Essential spectrum of recurrent operators}

In the proof of Theorem~\ref{The:GonLeMon} the {\em left-essential spectrum} of the operator $T:X\longrightarrow X$ plays a fundamental role even though in the statement just appears the {\em essential spectrum}. This happens because both sets coincide when $T$ is hypercyclic (see \cite{GonLeMon2000}). Here we prove that the same holds for recurrent operators and we will use it to prove Theorem~\ref{The:Banach-complex}.\\[-5pt]

Let us recall why the latter is true for hypercyclic operators. Following the general theory of {\em Fredholm} operators, see for instance \cite[Vol. I, 2.c]{LindTzaf1977}, we have that:
\begin{enumerate}[(a)]
	\item $\lambda \in \sigma_{\ell e}(T)$ if and only if $T-\lambda$ is not a {\em left-Fredholm} operator. Recall that a bounded linear operator $S \in \Lc(X)$ is {\em left-Fredholm} if the following conditions hold:
	\[
	\Ran(S) \text{ is closed} \qquad \text{and} \qquad \dim \Ker(S) < \infty.
	\]
	
	\item $\lambda \in \sigma_e(T)$ if and only if $T-\lambda$ is not a {\em Fredholm} operator. Recall that a bounded linear operator $S \in \Lc(X)$ is {\em Fredholm} if the following conditions hold:
	\[
	\Ran(S) \text{ is closed,} \qquad \dim \Ker(S)<\infty \qquad \text{and} \qquad \codim\Ran(S)<\infty.
	\]
\end{enumerate}
Clearly $\sigma_{\ell e}(T) \subset \sigma_e(T)$. Conversely, given $\lambda \in \sigma_e(T)$ for which $\Ran(T-\lambda)$ is dense we can deduce that $\lambda \in \sigma_{\ell e}(T)$. This is why we have the equality $\sigma_{\ell e}(T) = \sigma_e(T)$ for every hypercyclic operator $T$, since hypercyclicity implies that $\Ran(T-\lambda)$ is dense for every $\lambda \in \CC$ (see for instance \cite[Lemma~2.53]{GrPe2011}). For a recurrent operator $T$, the previous argument just gives that $\sigma_{\ell e}(T)\setminus\TT = \sigma_e(T)\setminus\TT$ since in this case $\Ran(T-\lambda)$ is dense when $\lambda \in \CC\setminus\TT$, but not necessarily dense when $\lambda \in \TT$ (see \cite[Proposition~2.14]{CoMaPa2014}).\newpage

However, the following argument allows us to show the complete equality between both sets for every recurrent operator:

\begin{lemma}\label{Lem:left.essential}
	Let $X$ be a complex Banach space and let $T \in \Lc(X)$. Suppose that $\Omega \subset \CC$ is a non-empty connected open set with the property that $\Omega\setminus\sigma_p(T^*)\neq\varnothing$. Then we have that
	\[
	\sigma_{\ell e}(T) \cap \cl{\Omega}\neq\varnothing \quad \text{ if and only if } \quad \sigma_e(T) \cap \cl{\Omega}\neq\varnothing.
	\]
\end{lemma}
\begin{proof}
	Since $\sigma_{\ell e}(T) \subset \sigma_e(T)$ we just consider the case $\sigma_e(T) \cap \cl{\Omega} \neq \varnothing$. Suppose then by contradiction that $\sigma_{\ell e}(T) \cap \cl{\Omega} = \varnothing$, thus
	\[
	\cl{\Omega} \subset \left\{ \lambda \in \CC : \Ran(T-\lambda) \text{ is closed} \text{ and } \dim \Ker(T-\lambda) < \infty \right\},
	\]
	so the index
	\[
	\ind(T-\lambda):= \dim \Ker(T-\lambda) - \codim\Ran(T-\lambda),
	\]
	is well defined as an element of $\ZZ \cup \{\pm\infty\}$ for every $\lambda \in \cl{\Omega}$ (see \cite[Vol I, 2.c]{LindTzaf1977}). Since the set of semi-Fredholm operators is a norm-open subset of $\Lc(X)$ and the index function is norm-discrete-continuous on it (see \cite[Theorem 2.2]{OSearcoid1992}), we deduce that there is a connected open set $U \subset \CC$ with $\cl{\Omega} \subset U$ and such that the index function
	\[
	\lambda \mapsto \ind(T-\lambda) \in \ZZ \cup \{\pm\infty\},
	\]
	is constant on $U$. Moreover, this value is finite since given any $\lambda \in \Omega \setminus \sigma_p(T^*)$, we get that $\Ran(T-\lambda)$ is closed, $T-\lambda$ has dense range, $\codim\Ran(T-\lambda)=0$, and then we have the equality $\ind(T-\lambda) = \dim\Ker(T-\lambda) \in \NN_0$. Finally, given $\mu \in \sigma_e(T) \cap \cl{\Omega} \subset U$ we were assuming that $\mu \notin \sigma_{\ell e}(T)$, but then
	\[
	\Ran(T-\mu) \text{ is closed} \qquad \text{and} \qquad \dim \Ker(T-\mu) < \infty,
	\]
	so we necessarily have that $\codim\Ran(T-\mu)=\infty$, and hence $\ind(T-\mu)=-\infty$, which yields a contradiction since we have checked that $\ind(T-\mu) \in \NN_0$.
\end{proof}

\begin{proposition}\label{Pro:left.essential}
	Let $X$ be a complex Banach space and let $T \in \Lc(X)$. If $\sigma_p(T^*)$ has empty interior, then we have the equality $\sigma_{\ell e}(T) = \sigma_e(T)$.
\end{proposition}
\begin{proof}
	We already know that $\sigma_{\ell e}(T) \subset \sigma_e(T)$. Suppose by contradiction that there exists some element $\lambda \in \sigma_e(T) \setminus \sigma_{\ell e}(T)$. Using that the set $\sigma_{\ell e}(T)$ is compact (and hence closed) we can find a connected open neighbourhood $\Omega$ of $\lambda$ such that $\sigma_{\ell e}(T) \cap \cl{\Omega} = \varnothing$. This is a contradiction with Lemma~\ref{Lem:left.essential} since $\Omega\setminus\sigma_p(T^*) \neq \varnothing$, by the empty interior assumption, and also because $\lambda \in \sigma_e(T) \cap \cl{\Omega}$.
\end{proof}

The previous result applies to recurrent operators:

\begin{corollary}\label{Cor:left.essential}
	Let $X$ be a complex Banach space and let $T \in \Lc(X)$. If $T$ is recurrent, then we have the equality $\sigma_{\ell e}(T) = \sigma_e(T)$.
\end{corollary}
\begin{proof}
	If $T \in \Lc(X)$ is a recurrent operator we have that $\sigma_p(T^*) \subset \TT$ by \cite[Proposition~2.14]{CoMaPa2014}. Proposition~\ref{Pro:left.essential} yields the result.
\end{proof}

We are now ready to prove Theorem~\ref{The:Banach-complex}.

\subsection{Proof of Theorem~\ref{The:Banach-complex}}

By definition (ii) $\Rightarrow$ (i). It is also direct that (ii) $\Rightarrow$ (iii) by the Banach-Steinhaus theorem applied to the family of operators $\left\{ T^{l_n}\res_E : n\in\NN \right\}$.\\[-5pt]

To see the implications (i) $\Rightarrow$ (iv) and (iii) $\Rightarrow$ (iv) we use the following fact originally proved in \cite[Proof of Theorem 4.1]{GonLeMon2000} for bounded operators on complex separable Banach spaces:

\begin{lemma}\label{Lem:divergent}
	If $\sigma_{\ell e}(T) \cap \cl{\DD} = \varnothing$, then every infinite-dimensional closed subspace $Z \subset X$ admits a vector $x \in Z$ such that $\lim_{n\to\infty} \|T^nx\| = \infty$.
\end{lemma}

Since $T$ is recurrent we have the equality $\sigma_e(T)=\sigma_{\ell e}(T)$ by Corollary~\ref{Cor:left.essential}. Then, if (i) or (iii) holds but $\sigma_e(T) \cap \cl{\DD} = \varnothing$ we arrive to a contradiction: the vector with divergent orbit obtained by Lemma~\ref{Lem:divergent} cannot be in a recurrent subspace, neither in the subspace described by statement (iii).\\[-5pt]

To finish the proof we show that (iv) $\Rightarrow$ (ii): suppose that (iv) holds and let $\lambda \in \sigma_e(T) \cap \cl{\mathbb{D}}$. By Corollary~\ref{Cor:left.essential} we have that $\lambda \in \sigma_{\ell e}(T)$ so that $T-\lambda$ is not left-Fredholm and we can apply \cite[Proposition~D.3.4]{BaMa2009} to get an infinite-dimensional closed subspace $E \subset X$ and a compact operator $R \in \Lc(X)$ such that $(T-R)\res_E = \lambda I\res_E$, which implies that
\[
\left\|(T-R)^n\res_E\right\| \leq 1 \quad \text{ for every } n \in \NN.
\]
From now we modify the proof of \cite[Lemma 8.16]{BaMa2009}: by assumption $T$ is quasi-rigid with respect to some sequence $(k_n)_{n\in\NN}$. For each $n \in \NN$ we can write
\[
T^{k_n} = (T-R)^{k_n} + R_n,
\]
where $R_n$ is a compact operator. By \cite[Lemma 8.13]{BaMa2009}, one can find a non-increasing sequence $(E_n)_{n\in\NN}$ of finite-codimensional (and hence infinite-dimensional) closed subspaces of $E$, such that
\[
\|R_n\res_{E_n}\| \leq 1 \quad \text{ for every } n \in \NN.
\]
Then
\[
\left\|T^{k_n}\res_{E_n}\right\| = \left\| [ (T-R)^{k_n} + R_n ]\res_{E_n} \right\| \leq \left\| (T-R)^{k_n}\res_{E_n} \right\| + \left\| R_n\res_{E_n} \right\| \leq 2 < \infty.
\]
By Theorem~\ref{The:Banach-sufficient} we get (ii) for some subsequence $(l_n)_{n\in\NN}$ of $(k_n)_{n\in\NN}$.\QEDh

\begin{remark}
	Corollary~\ref{Cor:hyp.iff.rec} is now proved in the \textbf{complex case} as a direct consequence of Theorems~\ref{The:GonLeMon} and \ref{The:Banach-complex}. However, in view of Lemma~\ref{Lem:divergent} there is an alternative proof without using Theorem~\ref{The:Banach-complex}:
	\begin{enumerate}[--]
		\item Actually, it is clear that having a hypercyclic subspace implies having a recurrent subspace since every hypercyclic vector is recurrent.
		
		\item On the other hand, Lemma~\ref{Lem:divergent} implies that: a necessary condition for a weakly-mixing operator to admit a recurrent subspace is that the left-essential spectrum must intersect the closed unit disc, since a vector with a divergent orbit is not recurrent. Hence, Theorem~\ref{The:GonLeMon} implies that the operator must have a hypercyclic subspace.
	\end{enumerate}
	In Section~\ref{Sec:6real} we obtain the \textbf{real case} of Corollary~\ref{Cor:hyp.iff.rec}.
\end{remark}

\subsection{Comments on Theorem~\ref{The:Banach-complex}}

Before the comments on Theorem~\ref{The:Banach-complex} let us introduce the following standard notation:

\begin{definition}
	Let $X$ be a Banach space:
	\begin{enumerate}[(a)]
		\item Given a Banach space $Y$ we will denote by $\Lc(X,Y)$ the {\em set of bounded linear operators} from $X$ to $Y$.
		
		\item Let $(Y_j)_{j\in J}$ and $\Kc \subset \bigcup_{j\in J} \Lc(X,Y_j)$ be families of Banach spaces and compact operators respectively. We say that a sequence $(x_n)_{n\in\NN} \subset X$ is {\em $\Kc$-null} if $\lim_{n\to\infty} Kx_n = 0$ for every operator $K \in \Kc$.
	\end{enumerate}
\end{definition}

As we did for Theorem~\ref{The:Banach-sufficient} in Subsection~\ref{SubSec:4.2comments}, we can relax the hypothesis of Theorem~\ref{The:Banach-complex}:

\begin{remark}[\textbf{Non-Separability}]\label{Rem:non-separability.2}
	Theorem~\ref{The:Banach-complex} is still true if $X$ is a {\em non-separable} space. This is not completely direct since the implications (i) $\Rightarrow$ (vii) and (vi) $\Rightarrow$ (vii) are deduced from Lemma~\ref{Lem:divergent} which is only valid for complex {\em separable} Banach spaces with the proof of \cite[Proof of Theorem 4.1]{GonLeMon2000}. Nevertheless, if one reads carefully the proof given for those implications in \cite[Section 8.3]{BaMa2009}, it can be deduced the following more general fact that uses directly the {\em essential spectrum}:
	
	\begin{proposition}\label{Pro:divergent}
		Let $X$ be a complex (and not necessarily separable) Banach space and let $T \in \Lc(X)$. If $\sigma_e(T) \cap \cl{\DD} = \varnothing$, then every infinite-dimensional closed subspace $Z \subset X$ admits a vector $x \in Z$ such that $\lim_{n\to\infty} \|T^nx\| = \infty$.
	\end{proposition}
	
	If we let $X$ and $(Y_n)_{n\in\NN}$ be complex Banach spaces, $T \in \Lc(X)$, and we consider any infinite-dimensional closed subspace $Z \subset X$, then Proposition~\ref{Pro:divergent} can be proved in three steps:
	\begin{enumerate}[--]
		\item \cite[Lemma 8.15 (b)]{BaMa2009} {\em If $\sigma_e(T) \cap \cl{\DD} = \varnothing$, then one can find $\lambda > 1$, $n_0 \in \NN$ and a countable family of compact operators $\Kc_0 \subset \Lc(X)$ such that the following holds for any normalized $\Kc_0$-null sequence $(e_j)_{j\in\NN} \subset X$:
		\[
		\liminf_{j\to\infty} \|T^ne_j\| \geq \lambda^n \quad \text{ for each } n \geq n_0.
		\]}
		
		\item \cite[Lemma 8.17]{BaMa2009} {\em If $Z$ is separable, then there exist a complex separable Banach space $\hat{Y}$ and a countable family of compact operators $\Kc_1 \subset \Lc(Z,\hat{Y})$ such that the following holds for any normalized $\Kc_1$-null sequence $(e_j)_{j\in\NN} \subset Z$: given any summable sequence of positive numbers $(\alpha_n)_{n\in\NN}$, there exists some vector $x \in \cl{\lspan_{\RR}\{e_j : j \in \NN\}} \subset Z$ such that
		\[
		\|T^nx\| \geq \alpha_n \limsup_{j\to\infty} \|T^ne_j\| \quad \text{ for each } n \in \NN.
		\]}
		
		\item \cite[Corollary 8.14]{BaMa2009} {\em If $\Kc = (K_n : Z \longrightarrow Y_n)_{n\in\NN}$ is a countable family of compact operators, then $Z$ contains a normalized $\Kc$-null basic sequence.}
	\end{enumerate}
	One just needs to apply \cite[Corollary 8.14]{BaMa2009} to the countable family of compact operators $\Kc = \Kc_0 \cup \Kc_1$ and choose a summable sequence $(\alpha_n)_{n\in\NN}$ for which $\alpha_n\lambda^n \to \infty$ when $n \to \infty$. The separability of $X$ is replaced by the separability of $Z$. It is worth mentioning that the vector $x \in Z$ constructed lies in the closure of the \textbf{real-linear} span of the sequence $(e_j)_{j\in\NN}$ selected. This fact will be important in the real case, see Section~\ref{Sec:6real}.
\end{remark}

\begin{remark}[\textbf{Local Quasi-Rigidity is not allowed in Theorem~\ref{The:Banach-complex}}]\label{Rem:local.q-r.2}
	We cannot repeat the exchange of assumption done for Theorem~\ref{The:Banach-sufficient} in Remark~\ref{Rem:local.q-r.1}: the infinite-dimensional closed subspace $E \subset X$ obtained in the proof by the non left-Fredholm condition of $T-\lambda$ cannot be controlled to be included in the closure of a set $Y \subset X$ such that $T^{k_n}x \to x$ for all $x \in Y$. In fact, let $X$ be any of the complex spaces $\ell^p(\NN)$, $1\leq p<\infty$, or $c_0(\NN)$, and let $B:X\longrightarrow X$ be the well-known {\em backward shift} operator on such a space $X$. Then for any fixed $\lambda \in \DD\setminus\{0\}$ consider the operator
	\[
	T := \lambda^{-1} B \oplus \lambda I : X \oplus X \longrightarrow X \oplus X,
	\]
	where $I:X\longrightarrow X$ is the identity operator on $X$. It is easy to verify that:
	\begin{enumerate}[1)]
		\item $T$ is not quasi-rigid neither recurrent since $(x,y) \in \Rec(T)$ implies that $y=0$;
		
		\item $\lambda^{-1} B$ is quasi-rigid since the Rolewicz operator is known to be (weakly-)mixing; 
		
		\item $\lambda \in \sigma_e(T) \cap \cl{\DD}$ since $\Ker(T-\lambda)$ is infinite-dimensional;
		
		\item $T$ has no recurrent subspace: otherwise $\lambda^{-1} B$ would have a recurrent subspace, but it is well-known that $\sigma_e(\lambda^{-1}B) = \sigma_{\ell e}(\lambda^{-1}B) = \lambda^{-1}\TT = \{ \mu\in\mathbb{C} : |\mu|=\frac{1}{|\lambda|} \}$, contradiction.
	\end{enumerate}
\end{remark}

\section{The real case}\label{Sec:6real}

In this section we extend Theorems~\ref{The:GonLeMon} and \ref{The:Banach-complex} to the real case via Theorem~\ref{The:Banach-real}. Since given a real-linear operator $T:X\longrightarrow X$ one of the equivalences included in Theorem~\ref{The:Banach-real} is:
\begin{enumerate}[--]
	\item {\em the essential spectrum of the complexification $\til{T}:\til{X}\longrightarrow\til{X}$ intersects the closed unit disk};
\end{enumerate}
we need some previous results allowing us to pass from subspaces for the real-linear system to subspaces for its complexification and vice versa.\\[-5pt]

The first of such results is: given a complex infinite-dimensional closed subspace $Z$ of the complexification $\til{X}$ we want to conclude that its projection
\begin{equation}\label{eq:projection}
	P(Z) := \{ x \in X : x+iy \in Z \text{ for some } y \in X \},
\end{equation}
contains an infinite-dimensional closed subspace. Note that $P(Z)$, expressed in \eqref{eq:projection} as the {\em real part projection} of $Z$, coincides with the {\em imaginary part projection} of $Z$, namely
\[
Q(Z) := \{ y \in X : x+iy \in Z \text{ for some } x \in X \}.
\]
Indeed, $P(Z)=Q(Z)$ because a vector $x+iy$ belongs to $Z$ if and only if $i\cdot(x+iy) = -y+ix$ belongs to $Z$ and also if and only if $(-i)\cdot(x+iy) = y-ix$ belongs to $Z$. Since $\til{X}$ can be identified with the direct sum space $X\oplus X$ we will prove the following more general fact:

\begin{lemma}\label{Lem:Bill}
	Let $(X,\|\cdot\|_X)$ and $(Y,\|\cdot\|_Y)$ be (real or complex) Banach spaces and consider an infinite-dimensional closed subspace $Z \subset X\oplus Y$ of its direct sum. If we denote by
	\[
	P_X : X\oplus Y \longrightarrow X \qquad \text{ and } \qquad P_Y : X\oplus Y \longrightarrow Y,
	\]
	the standard projections on the corresponding subspace, then at least one of the subspaces $P_X(Z)$ or $P_Y(Z)$ contains an infinite-dimensional closed subspace.
\end{lemma}
\begin{proof}
	Suppose that $P_Y(Z) \subset Y$ does not admit any infinite-dimensional closed subspace:
	
	\begin{claim}
		For every $\delta>0$ and every infinite-dimensional closed subspace $V \subset Z$ there exist a vector $(x,y) \in V$ such that $\|x\|_X=1$ and $\|y\|_Y<\delta$.
	\end{claim}
	\begin{proof}[Proof of the Claim]
		By the initial assumption on $P_Y(Z) \subset Y$ we just have two possibilities: the subspace $P_Y(V)$ is closed (and hence finite-dimensional) or $P_Y(V)$ is not closed. In both cases $P_Y\res_V : V \longrightarrow \cl{P_Y(V)}$ is not a Banach isomorphism so it is not bounded from below and considering the norm $\|(x,y)\|:=\max\{\|x\|_X,\|y\|_Y\}$ in the space $X\oplus Y$ we get the {\em Claim}.
	\end{proof}
	
	Now one can modify the Mazur's theorem (see \cite[Vol I, Theorem~1.a.5 and Lemma~1.a.6]{LindTzaf1977}) and easily construct a basic sequence $(x_n,y_n)_{n\in\NN} \subset Z$ with the properties:
	\begin{enumerate}[(a)]
		\item $\|x_n\|_X=1$ and $\|y_n\|_Y<\frac{1}{2^n}$ for every $n \in \NN$ (by using the {\em Claim});
		
		\item $(x_n)_{n\in\NN}$ is a basic sequence for $(X,\|\cdot\|_X)$ (adding to each step of \cite[Vol I, Lemma~1.a.6]{LindTzaf1977} the corresponding functionals of the type $(x^*,0) \in X^*\oplus Y^* = (X\oplus Y)^*$).
	\end{enumerate}
	We claim that $P_X(Z)$ contains the subspace $\cl{\lspan\{x_n : n \in \NN\}}$: given any convergent series $x = \sum_{n\in\NN} a_n x_n$ we have that $a=(a_n)_{n\in\NN} \in c_0(\NN)$ since $(x_n)_{n\in\NN}$ is a normalized sequence, and hence $\sum_{n\in\NN} a_n y_n$ is an absolutely convergent series to some vector $y \in Y$, i.e.\
	\[
	y = \sum_{n\in\NN} a_n y_n \in Y, \quad \text{ so } \quad (x,y) = \sum_{n\in\NN} a_n (x_n,y_n) \in Z,
	\]
	and finally $x = P_X(x,y) \in P_X(Z)$.
\end{proof}

\begin{remark}[\textbf{W. B. Johnson's Proof}]\label{Rem:Bill}
	Lemma~\ref{Lem:Bill} admits an equivalent shorter proof in terms of {\em strictly singular operators}: if both $P_X(Z)$ and $P_Y(Z)$ do not contain any infinite-dimensional closed subspace, then the operators $P_X:Z\longrightarrow X\oplus Y$ and $P_Y:Z\longrightarrow X\oplus Y$ are strictly singular. By \cite[Vol I, Theorem 2.c.5]{LindTzaf1977} we would have that $P_X+P_Y = I : Z \longrightarrow Z$, which is an isomorphism, has to be a strictly singular operator, contradiction.
\end{remark}

\begin{remark}\label{Rem:projection}
	From Lemma~\ref{Lem:Bill} one can easily deduce that: {\em given a real Banach space $X$ and $T \in \Lc(X)$, if its complexification $\til{T}$ admits an infinite-dimensional closed subspace $Z \subset \til{X}$ with some dynamical property among those described in Theorems~\ref{The:GonLeMon} or \ref{The:Banach-complex}, then $P(Z)$ as defined in \eqref{eq:projection}, admits an infinite-dimensional closed subspace with the same property}.
\end{remark}

\subsection{Divergent orbits for real-linear operators}

Once we know how to pass from a ``recurrent or hypercyclic subspace for $\til{T}$'' to one for $T$, we need to study the converse implication. We do it by proving a real version of Proposition~\ref{Pro:divergent}:

\begin{proposition}\label{Pro:real.divergent}
	Let $X$ be a real (and not necessarily separable) Banach space and let $T \in \Lc(X)$. If $\sigma_e(\til{T}) \cap \cl{\DD} = \varnothing$, then every infinite-dimensional closed subspace $E \subset X$ admits a vector $x \in E$ such that $\lim_{n\to\infty} \|T^nx\| = \infty$.
\end{proposition}

In order to prove Proposition~\ref{Pro:real.divergent}, we will rewrite the proof of Proposition~\ref{Pro:divergent}, but using the following real version of \cite[Corollary~8.14]{BaMa2009}:

\begin{lemma}\label{Lem:real.8.14}
	Let $X$ be a real Banach space, and let $E$ be an infinite-dimensional closed subspace of $X$. If $\Kc = (K_n : \til{E} \longrightarrow Y_n)_{n\in\NN}$ is a countable family of complex-linear compact operators, where each $Y_n$ is a complex Banach space, then there exist a normalized basic sequence $(e_j)_{j\in\NN} \subset E$ such that $(\til{e_j})_{j\in\NN} = (e_j+i0)_{j\in\NN} \subset \til{E}$ is a $\Kc$-null sequence.
\end{lemma}
This is just the {\em complexification} version of the well-known (real and complex) Banach spaces results \cite[Lemma~8.13 and Corollary~8.14]{BaMa2009}:
\begin{proof}[Proof of Lemma~\ref{Lem:real.8.14}]
	We claim that there exists a decreasing sequence $(E_j)_{j\in\NN}$ formed by finite-codimensional closed subspaces of $E$ such that $\|K_n\res_{\til{E_j}}\| \leq \frac{1}{2^j}$ whenever $n\leq j$.\\[-5pt]
	
	Since $K_1$ is compact, the adjoint operator $K_1^*$ is also compact, so one can find a finite number of functionals $z_1^*, ... , z_N^* \in \til{E}^*$ such that
	\[
	K_1^*(B_{Y_1^*}) \subset \bigcup_{1\leq k\leq N} B(z_k^* , \tfrac{1}{2}),
	\]
	where $B_{Y_1^*}$ is the closed unit ball of $Y_1^*$. Then we have that
	\begin{equation}\label{eq:normaK}
		\|K_1(z)\| = \sup\{ |z^*(z)| : z^* \in K_1^*(B_{Y_1^*}) \} \leq \max_{1\leq k\leq N} |z_k^*(z)| + \tfrac{1}{2}\|z\| \quad \text{ for all } z \in \til{E}.
	\end{equation}
	By \cite[Section~4.10]{MoMuPeSe2022} we can identify $(\til{E})^*$ with $\til{(E^*)}$, and in particular, for each $z_k^*$ there are $x_k^*,y_k^* \in E^*$ such that
	\[
	z_k^*(x+iy) = \left[ x_k^*(x) - y_k^*(y) \right] + i \left[ y_k^*(x) + x_k^*(y) \right],
	\]
	for every $x+iy \in \til{E}$ and $1\leq k\leq N$. Hence we have the inclusion
	\[
	\til{\Ker(x_k^*) \cap \Ker(y_k^*)} \subset \Ker(z_k^*) \subset \til{E} \quad \text{ for every } 1\leq k\leq N.
	\]
	Therefore, by \eqref{eq:normaK}, the finite-codimensional closed subspace
	\[
	E_1 := E \quad \cap \bigcap_{1\leq k\leq N} \Ker(x_k^*) \cap \Ker(y_k^*) \subset E,
	\]
	has the property $\|K_1\res_{\til{E_1}}\| < \frac{1}{2}$.\\[-5pt]
	
	Recursively, if we have $E_1,...,E_j$ already constructed we obtain $E_{j+1}$ in the same way: consider a finite covering of $\bigcup_{n=1}^{j+1} K_n^*(B_{Y_n^*})$ with balls of diameter lower than $\frac{1}{2^{j+1}}$ and intersect $E_j$ with the kernels of the corresponding finite sequence of functionals.\\[-5pt]
	
	Lastly, the Mazur's theorem (see \cite[Vol I, page 4]{LindTzaf1977} or \cite[Lemma C.1.1]{BaMa2009}) provides a normalized basic sequence $(e_j)_{j\in\NN} \subset E$ such
	that $e_j \in E_j$ for all $j \in \NN$. This sequence $(e_j)_{j\in\NN}$ has the required properties.
\end{proof}

\begin{proof}[Proof of Proposition~\ref{Pro:real.divergent}]
	Assume that $E$ is separable. Let $\Kc_0 \subset \Lc(\til{X})$ and $\Kc_1 \subset \Lc(\til{E},\hat{Y})$ be the countable families of compact operators obtained by \cite[Lemma~8.15~(b)]{BaMa2009} and \cite[Lemma~8.17]{BaMa2009} respectively. Apply Lemma~\ref{Lem:real.8.14} to the countable family of compact operators $\Kc = \Kc_0 \cup \Kc_1$ to obtain a normalized basic sequence $(e_j)_{j\in\NN} \subset E$ such that $(\til{e_j})_{j\in\NN} = (e_j+i0)_{j\in\NN} \subset \til{E}$ is a $\Kc$-null sequence. Choose a summable sequence $(\alpha_n)_{n\in\NN}$ for which $\alpha_n\lambda^n \to \infty$ and apply \cite[Lemma~8.17]{BaMa2009} to $(\til{e_j})_{j\in\NN}$ to get a vector
	\[
	z = x+iy \in \cl{\lspan_{\RR}\{\til{e_j} : j \in \NN\}} \quad \text{ for which } \quad \lim_{n\to\infty} \|\til{T}^nz\| = \infty.
	\]
	Note that $x \in E$, $y=0$ and hence
	\[
	\lim_{n\to\infty} \|T^nx\| = \lim_{n\to\infty} \|\til{T}^nz\| = \infty.\qedhere
	\]
\end{proof}

We are now ready to prove Theorem~\ref{The:Banach-real}.

\subsection{Proof of Theorem~\ref{The:Banach-real}}

We show (a) and (b) at the same time following the proofs of Theorems~\ref{The:GonLeMon} and \ref{The:Banach-complex}: firstly, we observe that (ii) $\Rightarrow$ (iii) by the Banach-Steinhaus theorem.\\[-5pt]

To see (i) $\Rightarrow$ (iv) and (iii) $\Rightarrow$ (iv) note that if (i) or (iii) hold but $\sigma_e(\til{T}) \cap \cl{\DD} = \varnothing$ we arrive to a contradiction: the vector with divergent orbit obtained by Proposition~\ref{Pro:real.divergent} cannot be in a recurrent subspace, neither in the subspace described by statement (iii).\\[-5pt]

Finally, since $T$ is weakly-mixing or quasi-rigid if and only if so is $\til{T}$, we have that statement (iv) implies (i), (ii) and (iii) for the complexification operator $\til{T}$ by Theorems~\ref{The:GonLeMon} and \ref{The:Banach-complex}. By Lemma~\ref{Lem:Bill} we deduce (i), (ii) and (iii) for $T$ itself (see Remark~\ref{Rem:projection}).\QEDh

\subsection{Comments on Theorem~\ref{The:Banach-real}}

\begin{remark}[\textbf{Non-Separability}]\label{Rem:non-separability.3}
	As in Theorems~\ref{The:Banach-sufficient} and \ref{The:Banach-complex}, the quasi-rigid part of Theorem~\ref{The:Banach-real} is still true for non-separable Banach spaces.
\end{remark}

\begin{remark}[\textbf{Operators without common hypercyclic subspaces}]
	It was first shown in \cite[Example~2.1]{ABLP2005} that there are two operators admitting hypercyclic subspaces which do not have a common hypercyclic subspace. The example was extended in \cite[Exercise~8.2]{BaMa2009}. In particular, given any $\lambda \in \CC\setminus\cl{\DD}$:
	\begin{enumerate}[--]
		\item {\em let $T \in \Lc(\ell^2(\NN))$ be a weakly-mixing operator, on the complex space $\ell^2(\NN)$, with a hypercyclic subspace. Then
			\[
			T_1:=T\oplus \lambda B : \ell^2(\NN)\oplus \ell^2(\NN) \longrightarrow \ell^2(\NN)\oplus \ell^2(\NN),
			\]
			and
			\[
			T_2:= \lambda B \oplus T : \ell^2(\NN)\oplus \ell^2(\NN) \longrightarrow \ell^2(\NN)\oplus \ell^2(\NN),
			\]
			are weakly-mixing operators with a hypercyclic subspace, but they do not have any common hypercyclic subspace (where $B:\ell^2(\NN)\longrightarrow\ell^2(\NN)$ is the backward shift)}.
	\end{enumerate}
	The proof is based on the following four facts:
	\begin{enumerate}[1)]
		\item $T_1$ and $T_2$ are weakly-mixing because $\lambda B$ is mixing (it is the Rolewicz operator);
		
		\item $\sigma_e(T_1) \cap \cl{\DD} \neq \varnothing$ and $\sigma_e(T_2) \cap \cl{\DD} \neq \varnothing$ because $\sigma_e(T) \cap \cl{\DD} \neq \varnothing$;
		
		\item for a Hilbert space $H$ there is a simple proof, using orthogonality, of the fact: {\em for any infinite-dimensional closed subspace $Z \subset H \oplus H$, at least one of the projections $P_1(Z)$ or $P_2(Z)$ admits an infinite-dimensional closed subspace} (see~\cite[Example~2.1]{ABLP2005});
		
		\item the Rolewicz operator $\lambda B$ has no hypercyclic subspace (see Remark~\ref{Rem:local.q-r.2}).
	\end{enumerate}

	Theorem~\ref{The:Banach-real} and Lemma~\ref{Lem:Bill} allow us to generalize these examples, and in particular, for any (real or complex) number $\lambda$ with $|\lambda|>1$:
	\begin{enumerate}[--]
		\item {\em let $T \in \Lc(X)$ be a weakly-mixing operator with a hypercyclic subspace on a (real or complex) Banach space $X$. Let $Y$ be the (real or complex) $\ell^p(\NN)$ ($1\leq p<\infty$) or $c_0(\NN)$ space and denote by $B:Y\longrightarrow Y$ the backward shift on $Y$. Then
			\[
			T_1:=T\oplus \lambda B : X\oplus Y \longrightarrow X\oplus Y \quad \text{ and } \quad T_2:= \lambda B \oplus T : Y\oplus X \longrightarrow Y\oplus X,
			\]
			are weakly-mixing operators with a hypercyclic subspace, but they do not have any common hypercyclic subspace}.
	\end{enumerate}
\end{remark}

\section{Further results, applications and open problems}\label{Sec:7applications}

One of the objectives of the {\em hypercyclic spaceability theory} is to establish sufficient conditions for an operator to admit a hypercyclic subspace. This is the case of Theorem~\ref{The:LeMon} which can be easily reproved with the theory developed here: for any operator $S = T-K \in \Lc(X)$ with $\|S\|\leq 1$ we have that its essential spectrum (which is a subset of the spectrum of $S$) is included in the closed unit disk, so any weakly-mixing compact perturbation $T=S+K$ of $S$ admits a hypercyclic subspace (use Theorem~\ref{The:Banach-real} for the real case). We can also recover the following well-known result (see \cite[Corollary~10.11]{GrPe2011}):

\begin{corollary}
	Let $X$ be a complex separable Banach space and let $T \in \Lc(X)$ be a weakly-mixing (resp.\ quasi-rigid) operator. If $\Ker(T-\lambda)$ is infinite-dimensional for some $\lambda \in \cl{\DD}$, then $T$ has a hypercyclic (resp.\ recurrent) subspace.
\end{corollary}
\begin{proof}
	In that case $\lambda \in \sigma_e(T) \cap \cl{\DD}$ and Theorem~\ref{The:Banach-complex} applies.
\end{proof}

Nonetheless, Theorem~\ref{The:LeMon} and the previous corollary are somehow restrictive: the first needs that the perturbed operator has norm bounded by $1$, and the second requires the existence of plenty of eigenvectors associated to the same eigenvalue. We propose an alternative ``sufficient condition result'' in order to obtain the existence of hypercyclic subspaces:

\begin{corollary}\label{Cor:C-t.sufficient}
	Let $X$ be a (real or complex) Banach space and let $T,S,K \in \Lc(X)$ where the operator $T$ is weakly-mixing, $S$ is quasi-rigid, $K$ is compact and $T=S+K$. Then $T$ has a hypercyclic subspace if and only if $S$ has a recurrent subspace.
\end{corollary}
\begin{proof}
	For the complex case note that $\sigma_e(T)=\sigma_e(S)$ since every compact perturbation preserves the essential spectrum, so Theorems~\ref{The:GonLeMon} and \ref{The:Banach-complex} yield the result. Use Theorem~\ref{The:Banach-real} for the real case of the result.
\end{proof}

The idea after Corollary~\ref{Cor:C-t.sufficient} is applying it to the {\em C-type operators}.

\subsection{Application to C-type operators}

In the last years lots of dynamical properties have been distinguished in Linear Dynamics by the so-called {\em C-type operators}. Using them one can construct operators which are {\em chaotic} but not {\em $\Uc$-frequently hypercyclic}, see \cite{Menet2017}; {\em chaotic} and {\em frequently hypercyclic} but that do not admit an {\em ergodic probability measure with full support}; {\em chaotic} and {\em $\Uc$-frequently hypercyclic} but not {\em frequently hypercyclic}; {\em chaotic} and {\em mixing} but not {\em $\Uc$-frequently hypercyclic}, see \cite{GriMaMe2021}; {\em invertible} and {\em frequently hypercyclic} (resp.\ {\em $\Uc$-frequently hypercyclic}) but whose inverses are not {\em frequently hypercyclic} (resp.\ {\em $\Uc$-frequently hypercyclic}), see \cite{Menet2019U-freq,Menet2019freq}.\\[-5pt]

In this section we show that every {\em C-type} operator as defined in \cite{Menet2017,GriMaMe2021} admits a hypercyclic subspace. We do it in two different ways:
\begin{enumerate}[--]
	\item first by using the essential spectrum and applying Theorem~\ref{The:Banach-complex};
	
	\item and secondly by constructing explicitly a sequence of subspaces to apply Theorem~\ref{The:Banach-sufficient}.
\end{enumerate}
However, in both cases we use Corollary~\ref{Cor:C-t.sufficient} to simplify the problem.\\[-5pt]

As defined in \cite{GriMaMe2021}, the {\em C-type operators} are chaotic (and hence weakly-mixing) compact perturbations of operators with a dense set of periodic points (and hence quasi-rigid). More precisely, each {\em C-type operator} is associated to four parameters $v$, $w$, $\varphi$ and $b$, where:
\begin{enumerate}[--]	
	\item $w = (w_j)_{j\in\NN}$ is a sequence of complex numbers which is both bounded and bounded from below, i.e.\ $0 < \inf_{j\in\NN} |w_j| \leq \sup_{j\in\NN} |w_j| <\infty$;
	
	\item $\varphi:\NN_0\longrightarrow\NN_0$ is a map such that $\varphi(0) = 0$, $\varphi(n) < n$ for every $n \in \NN_0$, and the set $\varphi^{-1}(l) = \{ n \in \NN_0 : \varphi(n) = l \}$ is infinite for every $l \in \NN_0$;
	
	\item $b = (b_n)_{n\in\NN_0}$ is a strictly increasing sequence of positive integers such that $b_0 = 0$ and $b_{n+1}-b_n$ is a multiple of $2(b_{\varphi(n)+1}-b_{\varphi(n)})$ for every $n \in \NN$;
	
	\item $v = (v_n)_{n\in\NN}$ is a sequence of non-zero complex numbers such that $\sum_{n\in\NN} |v_n| < \infty$.
\end{enumerate}

\begin{definition}[\cite{GriMaMe2021}]
	Let $(e_k)_{k\in\NN_0}$ be the canonical basis of the linear space
	\[
	c_{00}(\NN_0) := \left\{ (x_j)_{j\in\NN_0} \in \CC^{\NN_0} : \text{ there exists } j_0 \in \NN \text{ with } x_j=0 \text{ for all } j\geq j_0 \right\},
	\]
	i.e.\ $e_k=(\delta_{k,j})_{j\in\NN_0}$. The linear map $T_{w,b}$ on $c_{00}(\NN_0)$ associated to the data $w$ and $b$ is defined by:
	\[
	T_{w,b} \ e_k := \begin{cases}
		w_{k+1}e_{k+1}, & \text{ if } k \in [b_n,b_{n+1}-1[, \\ \\
		- \left( \prod_{j=b_n+1}^{b_{n+1}-1} w_j \right)^{-1} e_{b_n}, & \text{ if } k = b_{n+1}-1, \text{ for } n\geq 0.
	\end{cases}
	\]
	Moreover, the linear map $T_{w,\varphi,b,v}$ on $c_{00}(\NN_0)$ associated to the data $v, w, \varphi$ and $b$ given as above is defined by:
	\[
	T_{w,\varphi,b,v} \ e_k := \begin{cases}
		w_{k+1}e_{k+1}, & \text{ if } k \in [b_n,b_{n+1}-1[, n\geq 0, \\ \\
		v_ne_{b_{\varphi(n)}} - \left( \prod_{j=b_n+1}^{b_{n+1}-1} w_j \right)^{-1} e_{b_n}, & \text{ if } k = b_{n+1}-1, n\geq 1, \\ \\
		\left( \prod_{j=b_0+1}^{b_{1}-1} w_j \right)^{-1} e_0, & \text{ if } k=b_1-1.
	\end{cases}
	\]
	When they extend continuously to an $\ell^p(\NN_0)$ space (with $1\leq p<\infty$) the resulting continuous operators are still denoted by $T_{w,b}$ and $T_{w,\varphi,b,v}$ respectively, and this last operator $T_{w,\varphi,b,v}$ is called the {\em operator of C-type} on $\ell^p(\NN_0)$ associated to the data $w, \varphi, b$ and $v$.
\end{definition}

As it is shown in \cite[Lemma~6.2 and Proposition~6.5]{GriMaMe2021}: we have that $T_{w,\varphi,b,v} = T_{w,b} + K_{\varphi,v}$ where
\[
K_{\varphi,v} \ x := \sum_{n\in\NN} v_n x_{b_{n+1}-1} \cdot e_{b_{\varphi(n)}} \quad \text{ for each } x=(x_j)_{j\in\NN_0} \in \ell^p(\NN_0)
\]
is a compact operator. The {\em continuity} of $T_{w,b}$ (and hence that of $T_{w,\varphi,b,v}$) depends on the following condition
\begin{equation}\label{eq:C-type.cont}
	\inf_{n\in\NN_0} \prod_{j=b_n+1}^{b_{n+1}-1} |w_j| > 0,
\end{equation}
and the {\em chaotic} behaviour of $T_{w,\varphi,b,v}$ can be deduced whenever the following holds
\begin{equation}\label{eq:C-type.chaos}
	\limsup_{\underset{N \in \varphi^{-1}(n)}{N \to \infty}} |v_N| \prod_{j=b_N+1}^{b_{N+1}-1} |w_j| = \infty \quad \text{ for every } n\geq 0.
\end{equation}

Since $T_{w,b}$ has a dense set of periodic points it is quasi-rigid (see \cite[Proposition~2.9]{GriLoPe2022}), so that Corollary~\ref{Cor:C-t.sufficient} can be applied to every {\em C-type operator} $T_{w,\varphi,b,v}$ defined as above, i.e.\ studying if it admits a hypercyclic subspace is equivalent to study if $T_{w,b}$ admits a recurrent subspace. The rest of this subsection is devoted to prove the following statement:
\begin{enumerate}[--]
	\item {\em every C-type operator $T_{w,\varphi,b,v}$ satisfying \eqref{eq:C-type.chaos} has a hypercyclic subspace}.
\end{enumerate}
As we were advancing before, we will do it by showing that the respective operator $T_{w,b}$ always has a recurrent subspace, and we prove this in two different ways:\\

\textbf{First Option ($T_{w,b}$ has a recurrent subspace): via the essential spectrum.} In view of Theorem~\ref{The:Banach-complex} and Corollary~\ref{Cor:left.essential}, the operator $T_{w,b}$ has a recurrent subspace if and only if there exists some $\lambda \in \cl{\DD}$ fulfilling at least one of the following properties:
\begin{enumerate}[(a)]
	\item $\Ker(T_{w,b}-\lambda)$ is infinite-dimensional;
	
	\item $\Ran(T_{w,b}-\lambda)$ is not closed.
\end{enumerate}

With respect to (a), it can be checked that every element in the {\em point spectrum} of $T_{w,b}$ has to be an appropriated root of the unity, so in particular the following inclusions hold
\[
\sigma_p(T_{w,b}) \subset \TT \subset \cl{\DD}.
\]
Moreover, it is not hard to check that:
\begin{enumerate}[--]
	\item {\em the subspace $\Ker(T_{w,b}-\lambda)$ is infinite-dimensional for some $\lambda \in \TT$ \textbf{if and only if} there are infinitely many blocks $[b_n,b_{n+1}[$ for which $\lambda^{b_{n+1}-b_n}=-1$}.
\end{enumerate}
This is a problem for our proposes since the condition that $b_{n+1}-b_n$ has to be a multiple of $2(b_{\varphi(n)+1}-b_{\varphi(n)})$, which is an assumption on the parameter $b = (b_n)_{n\in\NN_0}$, usually fights against having infinitely many blocks with an appropriated length for $\lambda$: note that $\lambda^{b_{\varphi(n)+1}-b_{\varphi(n)}} = -1$ implies that $\lambda^{b_{n+1}-b_n} = 1$. Indeed, for the most interesting examples exhibited in \cite{Menet2017,GriMaMe2021}, there is no $\lambda \in \TT$ with $\Ker(T_{w,b}-\lambda)$ being infinite-dimensional.\\[-5pt]

Regarding property (b), we can show that $\Ran(T_{w,b})=\Ran(T_{w,b}-0)$ is not closed whenever the operator fulfills \eqref{eq:C-type.chaos}, and hence that $0 \in \sigma_e(T_{w,b}) \cap \cl{\DD}$. We first claim that $T_{w,b}$ is not invertible: otherwise $T_{w,b}^{-1}$ would act on $c_{00}(\NN_0)$ as:
\[
T_{w,b}^{-1} \ e_k := \begin{cases}
	\tfrac{1}{w_{k-1}}e_{k-1}, & \text{ if } k \in ]b_n,b_{n+1}-1], \\ \\
	- \left( \prod_{j=b_n+1}^{b_{n+1}-1} w_j \right) e_{b_{n+1}-1}, & \text{ if } k = b_n, \text{ for } n\geq 0.
\end{cases}
\]
However, since $T_{w,b}$ is assumed to fulfill \eqref{eq:C-type.chaos} and $(v_N)_{N\in\NN}$ is a summable sequence we deduce that
\begin{equation}\label{eq:T^{-1}.discontinuous}
	\sup_{n\in\NN_0} \prod_{j=b_n+1}^{b_{n+1}-1} |w_j| = \infty,
\end{equation}
which implies that $T_{w,b}^{-1} : c_{00}(\NN_0) \longrightarrow c_{00}(\NN_0)$ does not extend continuously to $\ell^p(\NN_0)$. Since $T_{w,b}$ is quasi-rigid it has dense range. Moreover, $T_{w,b}$ is one-to-one so the open mapping theorem implies that $T_{w,b}$ is not surjective (otherwise it would be invertible). We deduce that $\Ran(T_{w,b})$ is not closed, that $0 \in \sigma_e(T_{w,b}) \cap \cl{\DD}$ and that $T_{w,b}$ has a {\em recurrent subspace}. By Corollary~\ref{Cor:C-t.sufficient} the C-type operator $T_{w,\varphi,b,v}$ has a {\em hypercyclic subspace}.\newpage

The previous argument is a particular case of the following result, which was already used in \cite{LeMon2001} for the particular case of the {\em bilateral weighted shifts} (the spectrum of such operators is either an annulus or a disk, depending on whether $T$ is invertible or not):

\begin{corollary}\label{Cor:1-1.sufficient}
	Let $X$ be a (real or complex) separable Banach space and let $T \in \Lc(X)$ be a weakly-mixing (resp.\ quasi-rigid) operator. If $T$ is one-to-one but not invertible, then $T$ has a hypercyclic (resp.\ recurrent) subspace.
\end{corollary}
\begin{proof}
	We have that $T$ is a recurrent operator (weak-mixing implies quasi-rigidity, which implies recurrence in its turn) so $T$ has dense range. Then $\Ran(T)$ is not closed (otherwise $T$ would be invertible by the open mapping theorem and the one-to-one assumption). The previous implies that $0 \in \sigma_e(T) \cap \cl{\DD}$ (and $0 \in \sigma_e(\til{T}) \cap \cl{\DD}$ for the real case).
\end{proof}

\textbf{Second Option ($T_{w,b}$ has a recurrent subspace): by finding explicit subspaces.} In view of Theorem~\ref{The:Banach-sufficient} and since $T_{w,b}$ is quasi-rigid with respect to some increasing sequence $(k_n)_{n\in\NN}$, it is enough to find a decreasing sequence $(E_n)_{n\in\NN}$ of infinite-dimensional closed subspaces of $\ell^p(\NN_0)$ for which
\begin{equation}\label{eq:sufficient}
	\sup_{n\in\NN} \left\|T_{w,b}^{k_n}\res_{E_n}\right\| < \infty.
\end{equation}
We use \eqref{eq:T^{-1}.discontinuous} which again comes from \eqref{eq:C-type.chaos} and the fact that $(v_N)_{N\in\NN}$ is a summable sequence (and hence convergence to $0$). Fix $M \geq \sup_{j \in \NN} |w_j| > 0$ and construct recursively a strictly increasing sequence of natural numbers $(l_n)_{n\in \NN}$ such that
\[
(k_n-1) < (b_{l_n+1}-b_{l_n}) \quad \text{ and } \quad M^{k_n-1} \leq \prod_{j=b_{l_n}+1}^{b_{l_n+1}-1} |w_j| \quad \text{ for all } n \in \NN.
\]
This selection can be done since by \eqref{eq:T^{-1}.discontinuous} we can always choose a big enough $l_n \in \NN$. Now, for each $n \in \NN$ consider the closed subspace
\[
E_n := \cl{ \lspan\{ e_{b_{l_m+1}-1} : m \geq n \} }.
\]
Hence, for any fixed $n \in \NN$ and given $x = \sum_{m\geq n} a_m \cdot e_{b_{l_m+1}-1} \in E_n$ we have that
\[
T_{w,b}^{k_n}(x) = \sum_{m\geq n} a_m \cdot \frac{\prod_{j=b_{l_m}+1}^{b_{l_m}+k_n-1} w_j}{\prod_{j=b_{l_m}+1}^{b_{l_m+1}-1} w_j} \cdot e_{b_{l_m}+k_n-1},
\]
which implies that
\[
\left\|T_{w,b}^{k_n}(x)\right\|_p^p \leq \sum_{m\geq n} \left(|a_m| \cdot \frac{M^{k_n-1}}{\prod_{j=b_{l_m}+1}^{b_{l_m+1}-1} |w_j|} \right)^p \leq \sum_{m\geq n} |a_m|^p = \|x\|_p^p,
\]
so $\|T_{w,b}^{k_n}\res_{E_n}\| \leq 1$ and \eqref{eq:sufficient} holds. Theorem~\ref{The:Banach-sufficient} yields the existence of a {\em recurrent subspace} for the operator $T_{w,b}$. By Corollary~\ref{Cor:C-t.sufficient} the C-type operator $T_{w,\varphi,b,v}$ has a {\em hypercyclic subspace}.

\begin{example}
	The previous arguments, together with the work and examples exhibited in \cite{Menet2017,GriMaMe2021} allow us to construct (Devaney) {\em chaotic} operators on every $\ell^p(\NN_0)$-space, having a {\em hypercyclic subspace}, which can be chosen to be:
	\begin{enumerate}[--]
		\item not {\em $\Uc$-frequently hypercyclic}, see \cite{Menet2017};
		
		\item {\em frequently hypercyclic} but not {\em ergodic}, see \cite[Examples~7.7 and 7.19]{GriMaMe2021};
		
		\item {\em $\Uc$-frequently hypercyclic} but not {\em frequently hypercyclic}, see \cite[Example~7.11]{GriMaMe2021};
		
		\item {\em topologically mixing} but not {\em $\Uc$-frequently hypercyclic}, see \cite[Example~7.16]{GriMaMe2021}.
	\end{enumerate}
\end{example}

\subsection{Open problems}

The following open problems could be interesting:

\begin{question}\label{Q:KitChan}
	Can Theorem~\ref{The:Banach-sufficient} be proved via the K. Chan approach (see \cite{Chan1999,Chan2001})?
\end{question}

There also exists a {\em hypercyclic spaceability theory} for operators acting on Fr\'echet spaces. However, for Fr\'echet spaces we loose the tool of the essential spectrum, so that this theory is much weaker than the one presented here for Banach spaces. Indeed, and as far as we know, there does not exist a general characterization of the weakly-mixing operators acting on Fr\'echet spaces that admit a hypercyclic subspace, even though some sufficient conditions about the existence of hypercyclic subspaces, such as Theorem~\ref{The:Montes}, are still true in the Fr\'echet setting (see \cite[Chapter~10, Section~10.5]{GrPe2011}). In view of that we propose the problem:

\begin{question}\label{Q:Frechet}
	Is (a proper variation of) Theorem~\ref{The:Banach-sufficient} still true for Fr\'echet spaces?
\end{question}

When we use the term ``{\em proper variation}'' we refer to using, in the Fr\'echet setting, the proper notion of ``{\em equicontinuity}'' instead of the original ``{\em equiboundedness}'' used in this paper. In particular, we really believe that the works \cite{Menet2011,Menet2013,Menet2014} have developed a sufficiently powerful theory of basic sequences, on Fr\'echet spaces with a continuous norm, in order to extend the proof of Theorem~\ref{The:Banach-sufficient} to this more general class of spaces.\\[-5pt]

Note also that an affirmative answer for Question~\ref{Q:KitChan} could lead to a positive solution for Question~\ref{Q:Frechet}, simpler than the constructive proof showed in Section~\ref{Sec:4Basic}, by using the techniques related to left-multiplication operators (acting on the algebra of linear operators) and tensor products (see \cite{BoMaPe2004,BoGro2012} for the hypercyclic and frequently hypercyclic cases respectively).\\[-5pt]

The following and last problem seems to be more intriguing:

\begin{question}
	Is Corollary~\ref{Cor:hyp.iff.rec} still true for weakly-mixing operators on Fr\'echet spaces?
\end{question}

Outside the {\em weak-mixing} context, it is an open problem to characterize when a hypercyclic operator that is not weakly-mixing admits a hypercyclic subspace; see \cite[Question~7]{Gilmore2020}. With the theory developed in this paper we can at least show the following:

\begin{proposition}
	Let $X$ be a complex (resp.\ real) separable Banach space and let $T \in \Lc(X)$ be a hypercyclic but not weakly-mixing operator. The following statements are equivalent:
	\begin{enumerate}[{\em(i)}]
		\item $T$ has a recurrent subspace;
		
		\item there exists an infinite-dimensional closed subspace $E \subset X$ and an increasing sequence of integers $(l_n)_{n\in\NN}$ such that $T^{l_n}x \to x$ for all $x \in E$;
		
		\item[{\em(ii')}] there exists an infinite-dimensional closed subspace $E \subset X$ and an increasing sequence of integers $(l_n)_{n\in\NN}$ such that $T^{l_n}x \to 0$ for all $x \in E$;
		
		\item there exists an infinite-dimensional closed subspace $E \subset X$ and an increasing sequence of integers $(l_n)_{n\in\NN}$ such that $\sup_{n\in\NN} \|T^{l_n}\res_E\| < \infty$;
		
		\item the essential spectrum of $T$ (resp.\ $\til{T}$) intersects the closed unit disk $\cl{\DD}$.
		
	\end{enumerate}
\end{proposition}
\begin{proof}
	Recall that every hypercyclic operator is quasi-rigid by \cite[Proposition~2.9]{GriLoPe2022}. Then the standard equivalences (i) $\Leftrightarrow$ (ii) $\Leftrightarrow$ (iii) $\Leftrightarrow$ (iv) follow from Theorem~\ref{The:Banach-complex} and \ref{The:Banach-real}; the implication (ii') $\Rightarrow$ (iii) follows from the Banach-Steinhaus theorem; and (iv) $\Rightarrow$ (ii') by the proof of (iv) $\Rightarrow$ (ii) showed in Theorem~\ref{The:Banach-complex} combined with \cite[Proof of Theorem~10.29]{GrPe2011}, since by hypercyclicity there exist a dense set $X_0 \subset X$ and an increasing sequence $(k_n)_{n\in\NN}$ of positive integers such that $T^{k_n}x \to 0$ for all $x \in X_0$.
\end{proof}

\section*{Funding}

This work was supported by the Spanish Ministerio de Ciencia, Innovaci\'on y Universidades, grant number FPU2019/04094; and by MCIN/AEI/10.13039/501100011033/FEDER, UE Projects PID2019-105011GB-I00 and PID2022-139449NB-I00.

\section*{Acknowledgments}

The author would like to thank Juan B\`es, Antonio Bonilla and Karl Grosse-Erdmann for suggesting the topic; also thank Sophie Grivaux and Alfred Peris for their valuable advice and numerous readings of the manuscript. Moreover, the author wants to specially thank Sophie Grivaux for the personal communication of the proof of Lemma~\ref{Lem:Bill} exposed in Remark~\ref{Rem:Bill}. The result and this proof were originally presented by W. B. Johnson to Sophie Grivaux.

$\ $\\

\textsc{Antoni L\'opez-Mart\'inez}: Universitat Polit\`ecnica de Val\`encia, Institut Universitari de Matem\`atica Pura i Aplicada, Edifici 8E, 4a planta, 46022 Val\`encia, Spain.\\

e-mail: alopezmartinez@mat.upv.es\\

\end{document}